\documentclass[12pt,a4wide, reqno]{amsart}
\usepackage[utf8]{inputenc}
\usepackage{amsmath}
\usepackage{amsfonts}
\usepackage{amssymb}
\usepackage{listings}
\usepackage{enumitem}
\usepackage{amsthm}
\usepackage{hyperref}
\usepackage{cleveref}
\usepackage{empheq, bm}
\usepackage[left=2cm,right=2cm,top=1.6cm,bottom=1.5cm]{geometry}
\usepackage{graphicx}
\newtheorem{theorem}{Theorem}[section]
\newtheorem{lemma}{Lemma}

\theoremstyle{definition}

\newtheorem{corollary}[theorem]{Corollary}

\newtheorem{question}[theorem]{Question}
\theoremstyle{remark}
\newtheorem{remark}{Remark}

\numberwithin{equation}{section}
\usepackage{fullpage}
\usepackage{enumitem}
\usepackage{xcolor}
\lstdefinelanguage{Sage}{
  morekeywords={for, print, in, and, if, else, elif, def, return, lambda, import, from, as, while, break, continue},
  morekeywords={[2]GF, PolynomialRing, modulus, append},
  morecomment=[l]\#,
  morestring=[b]',
  morestring=[b]",
  sensitive=true
}
\begin{document}

\title{Orthomorphism Polynomials of degree $7$ over finite fields}
\author{Bhitali Kousik and Dhiren Kumar Basnet}
\address{department of mathematical sciences, tezpur univeristy, tezpur 784028, Assam, India}
\email{\textnormal{msp24013@tezu.ac.in, dbasnet@tezu.ernet.in }}

\begin{abstract}
    A polynomial $f\in\mathbb{F}_q[x]$ is called an \textit{orthomorphism polynomial} over $\mathbb{F}_q$ if both $f(x)$ and $f(x)-x$ are permutation polynomials over $\mathbb{F}_q.$ In this paper, we completely characterize all non-exceptional orthomorphism polynomials of degree $7$ over finite fields. We also establish the non-existence of degree $7$ orthomorphism polynomials over several finite fields. Consequently, we obtain a complete characterization of degree $7$ orthomorphism polynomials over finite fields of all characteristics except $7.$ In each case where such polynomials exist, we provide their explicit forms and determine their exact number.

\end{abstract}
 
\maketitle
\section{Introduction} 
Let $\mathbb{F}_q$ be the finite field of order $q,$ where $q=p^r,~p$ is a prime and $r\in\mathbb{N}$. We denote the multiplicative group $\mathbb{F}_q\setminus\{0\}$ by $\mathbb{F}_q^*.$ A polynomial $f\in \mathbb{F}_q[x]$ is called a \textit{permutation polynomial} (PP) over $\mathbb{F}_q$ if the associated map $f:c \mapsto f(c)$ permutes $\mathbb{F}_q.$ We call a polynomial $f\in \mathbb{F}_q[x]$ an \textit{orthomorphism polynomial} (OP) over $\mathbb{F}_q$ if both $f(x)$ and $f(x)-x$ are PP over $\mathbb{F}_q$ and a \textit{complete permutation polynomial} (CPP) if both $f(x)$ and $f(x)+x$ are PP over $\mathbb{F}_q.$ Clearly, $f(x)$ is an OP over $\mathbb{F}_q$ if and only if $f(x)-x$ is a CPP over $\mathbb{F}_q.$ These polynomial notions are closely related to orthomorphism and complete mappings of finite groups \cite{evans2, mann} and have applications in check digit system \cite{schulz, winterhof}, combinatorial design theory (especially construction of orthogonal Latin squares) \cite{evans, wanless} and cryptography \cite{muratovic, stanica}. For more background and applications of OPs and CPPs we refer to \cite{ wanless2, wanless3, winterhof2}. 

The classification of OPs and CPPs based on their degrees has been studied for several decades. In 1982, Neiderreiter and Robinson \cite{CPP1} gave a complete classification of all CPPs of degree $<6,$ as well as settled the degree $6$ case for finite fields of order coprime to $6.$ Since the distinction between the classification of CPPs and OPs is essentially a change of linear term, the work of Shallue and Wanless \cite{shallue}, which classified all degree $6$ OPs over finite fields of characteristic $2$ and $3$ can be considered as completing the classification of the remaining degree $6$ cases not covered in \cite{CPP1}. Thus, degree $7$ provides the next natural case in the study of low-degree OPs.

In addition to their classification results in degree $6$, Shallue et al. \cite{wanless} enumerated \textit{canonical orthomorphism polynomials} (OP with zero constant term \cite{wanless2}) over fields of order $q\leq 17.$ In particular, their computations show that there are $7260,\,6422$ and $4624$ degree $7$ OPs over $\mathbb{F}_{11},~\mathbb{F}_{13}$ and $\mathbb{F}_{17}$ respectively. Although these enumerations give the number of degree $7$ OPs over several finite fields of small order, their explicit forms and classification are yet undetermined. This naturally leads to the problem of determining the explicit structure of degree $7$ OPs, and more generally, determining over which fields such polynomials exist. 

A major step toward this problem was provided by Fan \cite{xfan}, who classified PPs of degree $7$ over finite fields of odd characteristic up to linear transformations.  More precisely, two polynomials $f,~g\in \mathbb{F}_q[x]$ are said to be \textit{related by linear transformation} or \textit{linearly related} if there exist $a,~b\in \mathbb{F}_q^*$ and $c,~d\in \mathbb{F}_q$ such that $$g(x)=af(bx+c)+d.$$

Fan's classification gives an exhaustive and non-redundant list of representatives for the linear equivalence classes of non-exceptional degree $7$ PPs over finite fields of odd characteristic. Recall that, a polynomial $f\in\mathbb{F}_q[x]$ is called \textit{exceptional} if it is a PP over $\mathbb{F}_{q^m}$ for infinitely many positive integers $m$; otherwise it is called non-exceptional. It is also known that any exceptional polynomial over $\mathbb{F}_q$ is a PP over $\mathbb{F}_q.$ Together with M\"{u}ller's classification of exceptional polynomials \cite{muller} and Li et al.'s classification of degree $7$ PPs over fields of even characateristic \cite{li}, Fan's classification provides a natural framework for the classification of degree $7$ OPs.

Although the property of being a permutation polynomial is preserved under linear transformations, the same is not true for the orthomorphism polynomial. For example, $f(x)=4x+1\in \mathbb{F}_5[x]$ is an OP over $\mathbb{F}_5,$ but $f(4x)$ or $4f(x)$ is not. Since Fan's classification is given up to linear transformations, an additional analysis is required to determine which PPs are OPs. For each representative $f$ in Fan's list, one has to determine the pairs $(\alpha,\,\beta)\in(\mathbb{F}_q^*)^2$ for which $\alpha f(\beta x)-x$ is again a PP. Using Fan's classification, this condition can be reduced to a finite system of equations over $\mathbb{F}_q.$ Once the admissible pairs $(\alpha,\beta)$ are determined (by a SageMath program), translations of the form $\alpha f(\beta x+\gamma)+\delta,\,\gamma,\,\delta\in\mathbb{F}_q$ produce the corresponding OPs. Also, by modifying the SageMath code used to compute the values of $\alpha$ and $\beta$, we can enumerate the exact number of such OPs.

The purpose of this paper is to carry out this process for degree $7$. We first obtain the complete classification of non-exceptional orthomorphism polynomials of degree $7$ over all finite fields. In particular, we show that for a finite field $\mathbb{F}_q$ of order $q>7$, a non-exceptional OP of degree $7$ exists if and only if $q\in\{11,\,13,\,17,\,19,\,25\}.$ For these fields, we explicitly determine the corresponding families of non-exceptional OPs as well as exceptional OPs. The resulting counts over $\mathbb{F}_{11},\,\mathbb{F}_{13}$ and $\mathbb{F}_{17}$ agree with the enumeration of Shallue and Wanless \cite{wanless}. We also establish the non-existence of degree $7$ OPs over certain fields. Precisely, when $q=p^m>7$ and $p\neq 7,$ we prove that an orthomorphism polynomial of degree $7$ exists over $\mathbb{F}_q$ if and only if $q\in\{11,\,13,\,17,\,19,\,25\}.$ In characteristic $7$, we give an explicit characterization of OPs of degree $7$ over $\mathbb{F}_{49}.$ In particular, the classification over $\mathbb{F}_{49}$ yields $3,937,640$ degree $7$ OPs. These results substantially extend the existing classification of OPs based on their degrees.


\section{Preliminaries}
In this section, we recall some results from Xiang Fan's work on classification of degree $7$ PPs over finite fields \cite{xfan} which will be instrumental in our subsequent work.
\begin{lemma}\cite[Proposition~11]{xfan}\label{l3}
    Any degree $7$ PP over $\mathbb{F}_q$ is linearly related to $x^7$ or to some $G(x)=x^7+\sum_{i=1}^{t}g_ix^i,$ where $1\leq t \leq 5$ and all $g_i\in \mathbb{F}_q$ satisfy the following criteria:
    \begin{enumerate}[label=\textup{(\arabic*)}]
        \item 
        $0\neq g_t\in \textup{CK}_q(7-t)$ and $g_{t-1}\in \textup{CI}_q(7-t).$
        \item 
        If $7\mid q,$ then $g_{t-1}=0.$
        \item 
        If $t=5$ and $g_4=0,$ then $g_2\in \{0\}\cup \textup{CI}_q(2).$ 
        \item 
        If $t=4$ and $g_3=0,$ then $g_2\in \{0\}\cup \textup{CI}_q(3).$
        \item 
        If $t=3,~g_2=0$ and $q\equiv 1 (\textup{mod}~4),$ then $g_1\in \{0\}\cup \textup{CI}_q(2),$
    \end{enumerate}
    where $\textup{CK}_q(m)$ and $\textup{CI}_q(m),~m\in\mathbb{Z}$ were defined in \cite{xfan} as follows:\\
    \hspace*{3 cm} $\textup{CK}_q(m)=\{\theta^i: 0\leq i < \textup{gcd}(m,~q-1)\},$~and\\
 \hspace*{3 cm} $\textup{CI}_q(m)=\{\theta^j: 0\leq j < (q-1)/\textup{gcd}(m,~q-1)\},$ $\theta$ is a generator of $\mathbb{F}_q^*.$

Also, whenever there exists another polynomial $F(x)=x^7+\sum_{i=1}^{s}f_i x^i$ with $1\leq s \leq 5,~s_i\in \mathbb{F}_q$ satisfying the same requirements \emph{(1) \textendash\ (5)}, $G$ and $F$ are linearly related if and only if $G=F.$
\end{lemma}

The following lemma, one of the key results of \cite{xfan}, provides and essential foundation of our work.
\begin{lemma}\cite[Theorem~13]{xfan}\label{l4}
    Let $A=\{11,~13,~17,~19,~23,~25,~31\}.$ A non-exceptional PP over $\mathbb{F}_q$ of characteristic $p\notin\{2,~3,~7\}$ and $q>7$ exists if and only if $q\in A.$ Furthermore, a polynomial $F(x)\in\mathbb{F}_q[x]$ is a non-exceptional PP of degree $7$ over $\mathbb{F}_q$ if and only if it is linearly related to some $x^7+\sum_{i=1}^{5}f_ix^i$ with $(f_5,~f_4,~f_3,~f_2,~f_1)\in \mathbb{F}_q^5$ listed as follows. This list is exhaustive and non-redundant, meaning it provides a one-to-one correspondence between all the tuples $(f_5,~f_4,~f_3,~f_2,~f_1)$ and all the linearly related classes of non-exceptional PPs of degree $7$ over $\mathbb{F}_q.$
   \begin{displaymath}
 \begin{matrix*}[l]
 q=11:& (0,~0,~0,~5,~0),&(0,~0,~0,~8,~0),& (0,~1~,0,~0,~4),&(0,~1,~0,~8,~5), \\
 &(0,~1,~0,~9,~5),& (1,~0,~5,~0,~2),&(1,~0,~5,~0,~7),&(1,~0,~5,~8,~6)\\
 &(1,~1,~5,~2,~5),&(1,~2,~5,~9,~8),& (1,~4,~5,~8,~0),&(1,~8,~5,~8,~0),\\
 &(1,~5,~5,~1,~5),&(2,~0,~9,~0,~8),& (2,~0,~9,~0,~9),& (2,~0,~9,~4,~4),\\
 &(2,~1,~9,~5,~3),& (2,~2,~9,~5,~8),&(2,~4,~9,~8,~3),&(2,~8,~9,~5,~2), \\
&(2,~8,~9,~7,~8),&(2,~8,~9,~8,~1), &(2,~5,~9,~5,~2),& (2,~5,~9,~6,~1),\\
&(2,~5,~9,~8,~3).& & & \\
q=13:& (0,~0,~0,~0,~2), &(0,~0,~0,~0,~6), & (0,~0~,2,~0,~8),&(0,~0,~4,~0,~4),\\ & (0,~0,~8,~0,~3),& (0,~1,~0,~0,~2) &(0,~1,~1,~10,~5),&(0,~1,~2,~1,~0),\\
& (0,~1,~2,~3,~9), &(0,~1,~8,~7,~11),& (0,~2,~1,~0,~8), &(0,~4,~0,~0,~6),\\
& (0,~4,~1,~7,~1), &(0,~4,~4,~3,~3). & &  \\
q=17:& (0,~1,~10,~0,~16), &(1,~0,~6,~0,~11), & (1,~0,~7,~0,~0),&(1,~0,~13,~0,~7),\\
& (1,~0,~13,~0,~14),& (1,~0,~14,~0,~3) &(1,~3,~13,~11,~10),&(1,~10,~3,~14,~11)\\
& (3,~0,~7,~0,~4), &(3,~0,~10,~0,~14),& (3,~0,~12,~0,~0), &(3,~0,~14,~0,~8),\\
& (3,~0,~15,~0,~2), &(3,~9,~11,~14,~10), & (3,~9,~12,~14,~5),& (3,~9,~15,~14,~12),\\
&(3,~15,~10,~12,~4).&  &  & \\
q=19:& (0,~0,~0,~0,~16), &(1,~0,~3,~14,~11), & (1,~0~,5,~0,~4),&(1,~0,~7,~0,~11),\\
& (1,~0,~11,~0,~16),& (1,~0,~18,~9,~4) &(2,~0,~14,~0,~5),  &(2,~0,~16,~0,~9)\\
& (2,~0,~17,~0,~5). & & & \\
q=23:& (1,~1,~0,~4,~9), &(1,~5,~11,~5,~9), & (1,~2~,6,~19,~21).&  \\
q=25:& (0,~0,~0,~0,~\theta), &(0,~0,~0,~0,~\theta^5), & (\theta,~0~,\theta^2,~0,~0),& with~\theta=\theta^2+2~in~ \mathbb{F}_{25}.  \\
q=31:& (1,~0,~16,~0,~2), &(1,~17,~25,~25,~29), & (3,~1~,14,~19,~10).&   \\ 
\end{matrix*}
\end{displaymath}
\end{lemma}
\begin{remark} \label{r1}
    Fan also listed all the exceptional polnomials of degree $7$ over $\mathbb{F}_q$ for the above values of $q,$ up to linear transformations, (which was previously classified by M\"{u}ller \cite[Theorem~4]{muller}) as $x^7+\sum_{i=1}^{5}g_ix^i$ with $(g_5,~g_4,~g_3,~g_2,~g_1)\in \mathbb{F}_q^5$ as follows.
    \begin{displaymath}
        \begin{matrix*}[l]
        q=11:&(0,~0,~0,~0,~0),& (1,~0,~5,~0,~9), & (2,~0,~9,~0,~6).\\
        q=13:&(0,~0,~0,~0,~0).&  & \\
        q=17:&(0,~0,~0,~0,~0),& (1,~0,~10,~0,~8), & (3,~0,~5,~0,~11).\\
        q=19:&(0,~0,~0,~0,~0),& (1,~0,~3,~0,~7), & (2,~0,~12,~0,~18).\\
        q=23:&(0,~0,~0,~0,~0),& (1,~0,~20,~0,~8), & (5,~0,~17,~0,~11).\\
         q=25:&(0,~0,~0,~0,~0),& (1,~0,~1,~0,~4), & (\theta,~0,~\theta^2,~0,~4\theta^3).\\
        q=31:&(0,~0,~0,~0,~0),& (1,~0,~18,~0,~19), & (3,~0,~7,~0,~17).\\
       
        \end{matrix*}
    \end{displaymath}
\end{remark}
\begin{remark}\label{r2}
From \Cref{l3}, we know that any PP $H(x)=\sum_{i=0}^{7}h_ix^i$ over $\mathbb{F}_q$ with $h_7\neq 0$ can be normalized into $F(x)=x^7+\sum_{i=1}^{5}f_ix^i$ that satisfies all the criteria of \Cref{l3}. Therefore, by Lemma~\ref{l4}, whenever $q\in \{11,~13,~17,~19,~23,~25,~31\},~H$ is a PP over $\mathbb{F}_q$ if and only if $F$ is listed in \Cref{l4} or \Cref{r1} and such a $F$ is unique.
\end{remark}

We conclude this section with the following trivial lemma which will be used in the next section.
\begin{lemma} \cite{CPP1} \label{l5}
    If $f\in\mathbb{F}_q[x]$ is an orthomorphism polynomial over $\mathbb{F}_q,$ then for all $c,~d\in\mathbb{F}_q$ the same holds for $f(x+ c)+ d$ also.
\end{lemma}

\section{Main Results}
In this section we present our results on classification of degree $7$ OPs over finite fields. We begin with the following simple lemma, which will be needed later to enumerate the total number of OPs.

\begin{lemma}\label{l6}
    Let $F(x)= x^7+\sum_{i=1}^{5}f_ix^i$ be a polynomial over $\mathbb{F}_q.$ Define $g(x)=\alpha F(\beta x)$ and $h(x)=\alpha_1 F(\beta_1 x),$ where $\alpha,~\beta,~\alpha_1,~\beta_1 \in \mathbb{F}_q^*.$~If $g( x+\gamma)+\delta= h(x+\gamma_1)+\delta_1$ for all $x\in \mathbb{F}_q$ with $\gamma,~\delta,~\gamma_1,~\delta_1 \in \mathbb{F}_q,$ then we must have $\gamma=\gamma_1$ and $\delta=\delta_1.$ 
\end{lemma}
\begin{proof}
Equating the coefficients of all powers of $x$ in the given equality with some simplification yields the required result. Details are left to the reader.

\end{proof}
\subsection{Non-exceptional OPs over fields of characteristic $p\notin\{2,\,3,\,7\}$}
We first focus on determining the cases in which non-exceptional OPs can occur, starting with fields of characteristic $p\notin\{2,\,3,\,7\}.$ In each such case, we further investigate the exceptional OPs and determine their explicit forms. By Lemma~\ref{l4}, it then suffices to check for $q\in\{11,\,13,\,17,\,19,\,23,\,25,\,31\}$ only. For ease of calculation, we present our discussions step by step, according to the values of $q.$ The main methodology is detailed for the case $q=13.$ Similar approach applies to other values of $q$, but to avoid repetition, we do not elaborate on each case here.
\vspace{0.5 cm}
\\
\underline{\Large \textbf{$q=13$~:}}
\vspace{0.2cm}
\\
We combine \Cref{l4} and \Cref{r1} to relist all PPs of the prescribed form $f(x)=x^7+\sum_{i=1}^{5}a_ix^i$ over $\mathbb{F}_{13}$ up to linear transformations in terms of the families of tuples $(a_5,~a_4,~a_3,~a_2,~a_1)\in\mathbb{F}_{13}^5$ as follows:
\begin{equation}
\begin{matrix*}[l]
     \textup{Family 1}:(0,~0,~0,~0,~2), & \textup{Family 2}:(0,~0,~0,~0,~6),\\
     \textup{Family 3}: (0,~0,~2,~0,~8),&\textup{Family 4}:(0,~0,~4,~0,~4), \\
     \textup{Family 5}:(0,~0,~8,~0,~3),& \textup{Family 6}: (0,~1,~0,~0,~2),\\
     \textup{Family 7}:(0,~1,~1,~10,~5),& \textup{Family 8}:(0,~1,~2,~1,~0),\\
 \textup{Family 9}:(0,~1,~2,~3,~9),& \textup{Family 10}:(0,~1,~8,~7,~11),\\
 \textup{Family 11}:(0,~2,~1,~0,~8), & \textup{Family 12}:(0,~4,~0,~0,~6),\\
 \textup{Family 13}:(0,~4,~1,~7,~1),& \textup{Family 14}:(0,~4,~4,~3,~3), \\
 \textup{Family 15}:(0,~0,~0,~0,~0). &  \\
\end{matrix*}
\label{eq1}
\end{equation}

For example, Family 1 represents the polynomial $f_1(x)=x^7+2x.$ We first check whether $f_1$ is an OP or not. For this, it will be sufficient to check whether $h_1(x)=f_1(x)-x=x^7+x$ is a PP or not. According to \Cref{r2}, $h_1$ is a PP if and only if it is linearly related to one of the families listed above. As it satisfies all requirements (1) \textendash (5) of \Cref{l3}, it follows that $h_1,$ represented by the tuple $(0,~0,~0,~0,~1)$ must be equal to one of these families. Clearly, this is not possible and therefore, $f_1$ can not be an OP.

Next, we consider Family 2, which represents the polynomial $f_2(x)=x^7+6x.$ Define $h_2(x)$ as $f_2(x)-x=x^7+5x.$ Does it satisfy the requirements of \Cref{l3}? As we know that $2$ is a generator of $\mathbb{F}_{13}^*,$ so by definition of $\textup{CK}_q(m)$, we have
\begin{equation*}
    \begin{aligned}
        \text{CK}_{13}(7-1)=\text{CK}_{13}(6)&=\{2^i: 0\leq i < \text{gcd}(6,~12)\},\\
        &=\{1,~2,~3,~4,~6,~8\}.
    \end{aligned}
\end{equation*}
As $5\notin\text{CK}_{13}(7-1),$ it follows that $h_2$ does not satisfy condition (1) of \Cref{l3} and hence it need not be equal to any of the families of those tuples even if $h_2$ is a PP. \\The immediate next step is to determine whether $h_2$ is linearly related to any of these families. Suppose that $h_2(x)=af(bx+c)+d,$ where $~a,b\in\mathbb{F}_{13}^*;~c,~d\in\mathbb{F}_{13}$ and $f(x)=x^7+\sum_{i=1}^{5}a_ix^i$ is a polynomial over $\mathbb{F}_{13}$ represented by one of those families listed in (\ref{eq1}). We compare the coefficients of all powers of $x$ on both sides of this equality to obtain
$$1=ab^7,~0=c=aa_5b^5=aa_4b^4=aa_3b^3=aa_2b^2,~5=aa_1b~\text{and}~0=d.$$
Clearly, $a_5=a_4=a_3=a_2=0$ and $a_1\neq 0.$ Therefore, if $h_2$ is a PP, then it must be linearly related to one of the families of the form $(0,~0,~0,~0,~a_1)$ with $a_1\in\mathbb{F}_{13}^*.$ By observing the list (\ref{eq1}), it is evident that $a_1$ can be either $2$ or $6.$ Finally it requires to solve the following system of equations over $\mathbb{F}_{13}^*:$
\begin{equation}
    \begin{matrix*}[l]
         ab^7=1, & 5=2ab,
    \end{matrix*}
\end{equation}
and
\begin{equation}
    \begin{matrix*}[l]
         ab^7=1, & 5=6ab.
    \end{matrix*}
\end{equation}
These equations can be solved either manually or with the aid of a computer program (the code is provided in the next page) and we find none of them to be solvable over $\mathbb{F}_{13}^*.$ This means that $h_2$ is not a PP, and consequently $f_2$ cannot be an OP. 

Using similar processes, we went through all the remaining families but could not find any OPs. This observation led us to investigate their linearly related classes. For each family $f$ we only need to determine the values of $\alpha$ and $\beta$ belonging to $\mathbb{F}_{13}^*$ for which $\alpha f(\beta x)$ is an OP. By \Cref{l5}, each such pair $(\alpha,~\beta)$ yields an OP $\alpha f(\beta (x+\gamma))+\delta$, for all $\gamma,~\beta\in\mathbb{F}_{13}.$ We begin this process with Family 1.
\vspace{0.2 cm}\\
\textbf{Family 1:}~$(0,~0,~0,~0,~2)$\\ 
 We define 
$$f_1(x)=x^7+2x,~\textup{the polynomial represented by Family 1},$$
$$g_1(x)=\alpha f_1(\beta x)=\alpha \beta^7x^7+2\alpha \beta x,~~~\text{and}$$
$$h_1(x)=\alpha f_1(\beta x)-x=\alpha \beta^7x^7+(2\alpha \beta-1) x.$$
Suppose $h_1$ is a PP. Then it must be linearly related to one of the families of  (\ref{eq1}), say $f'$. It means that there must exist $a,~b\in\mathbb{F}_{13}^*$ and $c,~d\in\mathbb{F}_{13}$ such that
\begin{equation*}
    \begin{aligned}
        h_1(x)&=af'(bx+c)+d
        \end{aligned}
\end{equation*}
which implies that,
\begin{equation*}
    \begin{aligned}
        \alpha \beta^7x^7+(2\alpha \beta-1) x&=a(bx+c)^7+a\sum_{i=1}^{5}a_i(bx+c)^i+d.
        \end{aligned}
\end{equation*}
Comparing the coefficients of all powers of $x,$ we get
$$\alpha \beta^7=ab^7,~0=c=aa_5b^5=aa_4b^4=aa_3b^3=aa_2b^2,~2\alpha \beta-1=aa_1b~\text{and}~0=d.$$
It follows from these conditions that $f'$ must come from the families of the form $(0,~0,~0,~0,~a_1)$ with $a_1\in\mathbb{F}_{13},$ namely Family 1 $(a_1=2),$ Family 2 $(a_1=6)$ and Family 15 $(a_1=0).$ Ultimately, our problem reduces to solving the following systems of equations:
\begin{equation}\label{eq4}
    \begin{matrix*}[l]
         \alpha \beta ^7=ab^7, & 2\alpha \beta -1=2ab,
    \end{matrix*}
\end{equation}
\begin{equation}\label{eq5}
    \begin{matrix*}[l]
         \alpha \beta ^7=ab^7, & 2\alpha \beta -1=6ab,
    \end{matrix*}
\end{equation}
and
\begin{equation}\label{eq6}
    \begin{matrix*}[l]
         \alpha \beta ^7=ab^7, & 2\alpha \beta -1=0.
    \end{matrix*}
\end{equation}
 To exhaust all the OPs from this family, we only need to determine the values of $\alpha$ and $\beta.$ We run the following code \textbf{OP7(13)} in SageMath to find all the possible values of $a,~b,~\alpha$ and $\beta.$~Let us start by working on (\ref{eq4}).
 \vspace{0.5 cm}
 \hrule
 \vspace{0.2 cm}
 \textbf{OP7(13)}:
 \vspace{0.2 cm}
 \vspace{0.1 cm}
 \begin{lstlisting}
 # Define base field and extension
F13 = GF(13)

# Find all nonzero elements of F13
nonzero = [z for z in F13 if z != 0]

# Store all solutions
solutions = []

# Step 1: Brute-force all (a, b, A, B) in (F13^*)^4. We use A and B in place of alpha and beta respectively
for a in nonzero:
    for b in nonzero:
        ab7 = a * b^7
        ab = a * b
        for A in nonzero:
            for B in nonzero:
                AB7 = A * B^7
                AB = A * B
                # Check both conditions of equation (3.4)
                if AB7 == ab7 and 2 * AB - 1  == 2 * ab:
                    solutions.append((a, b, A, B))

# Print all solutions
print(f"\nTotal tuples of the form (a, b, A, B) found: {len(solutions)} and they are \n")

for i, (a, b, A, B) in enumerate(solutions, 1):
    print(f"{i}: ({a}, {b}, {A}, {B})")

# Step 2: From all those solutions avoid repetitive polynomials
seen_AB = set()
final_solutions = []

for a, b, A, B in solutions:
    AB_pair = (A*B^7, A*B)
    if AB_pair not in seen_AB:
        seen_AB.add(AB_pair)
        final_solutions.append((a, b, A, B))

# Print the final solutions
print(f"\nFinal values of (A, B):\n")
for i, (a, b, A, B) in enumerate(final_solutions, 1):
    print(f"{i}: (A = {A}, B = {B})")

print(f"\nTotal number of unique solutions: {len(final_solutions)}")
\end{lstlisting}
The output of this program from Step 1 consists of $72$ tuples of the form $(a,~b,~\alpha,~\beta).$ Among them are $(1,~3,~2,~5),~(1,~3,~5,~2),~(1,~3,~6,~2),~(3,~1,~2,~5)$ and many more. Clearly, the first and fourth tuples yield the same pair $(\alpha,~\beta)$ and thus the same polynomial. Although the first and second tuples give us distinct pairs of $(\alpha,~\beta),$ the corresponding values of $a$ and $b$  indicate that they represent the same polynomial $g_1(x)=\alpha \beta^7x^7+2\alpha \beta x=ab^7x^7+(2ab+1)x;~a=1,~b=3.$ ~Step 2 is utilized at this point to discard all tuples producing the same values for both $\alpha\beta^7$ and $\alpha\beta,$ retaining only one representative from each such tuple. The final values of $(\alpha,~\beta)$ obtained from Step 2 are $(2,~5)$ and $(1,~10)$ and both these pairs give rise to two distinct orthomorphism polynomials which are $2f_1(5x)=3x^7+7x$ and $f_1(10x)=10x^7+10x.$ 

We apply similar processes to both systems of equations (\ref{eq5}) and (\ref{eq6}). In each case, orthomorphism polynomials are found to exist. Specifically, (\ref{eq5}) yields four pairs of $(\alpha,~\beta)$, whereas from (\ref{eq6}) we obtain two of such pairs. According to \Cref{l5}, each such pair $(\alpha,~\beta)$ generates OPs of the form $\alpha f_1(\beta(x+\gamma))+\delta,~\text{for all choices of}~\gamma,~\delta \in \mathbb{F}_{13}.$ Furthermore, \Cref{l6} ensures that all these polynomials are distinct. Corresponding to each such pair $(\alpha,~\beta)$, there are $13$ possible choices for both $\gamma$ and $\delta.$ Since there are eight pairs in total, this yields $8\times 13\times 13=1352$ distinct OPs from Family 1. 

This forms our main methodology. For each of the remaining families over $\mathbb{F}_{13}$ we apply the same technique with minor changes to the code to obtain our results. For brevity, we specify only the number of pairs $(\alpha,~\beta)$ associated with each family, whenever they exist along with those corresponding systems of equations which are solvable over $\mathbb{F}_{13}^*.$ The explicit forms of these pairs will be presented in the subsequent theorem.
\vspace{0.3 cm}\\
\textbf{Family 2:}~$(0,~0,~0,~0,~6)$ 
\begin{equation}\label{eq8}
    \begin{matrix*}[l]
         \alpha \beta ^7=ab^7, & 6\alpha \beta -1=2ab,
    \end{matrix*}
\end{equation}
\begin{equation}\label{eq9}
    \begin{matrix*}[l]
         \alpha \beta ^7=ab^7, & 6\alpha \beta -1=6ab,
    \end{matrix*}
\end{equation}
\begin{equation}\label{eq10}
    \begin{matrix*}[l]
         \alpha \beta ^7=ab^7, & 6\alpha \beta -1=0.
    \end{matrix*}
\end{equation}
\begin{itemize}
    \item 
    (\ref{eq8}) : $4$ pairs, (\ref{eq9}) : $2$ pairs,  (\ref{eq10}) : $2$ pairs.
\end{itemize}
\textbf{Family 3:}~$(0,~0,~2,~0,~8)$ 
\begin{equation}\label{eq11}
    \begin{matrix*}[l]
         \alpha \beta ^7=ab^7, & 2\alpha \beta^3=2ab^3, & 8\alpha\beta -1=8ab,
    \end{matrix*}
\end{equation}
\begin{itemize}
    \item 
    (\ref{eq11}) : $6$ pairs.
\end{itemize}
\textbf{Family 4:}~$(0,~0,~4,~0,~4)$ 
\begin{equation}\label{eq15}
    \begin{matrix*}[l]
        \alpha \beta ^7=ab^7, & 4\alpha \beta^3=4ab^3, & 4\alpha\beta -1=4ab,
    \end{matrix*}
\end{equation}
\begin{itemize} 
    \item 
    (\ref{eq15}) : $6$ pairs.
\end{itemize}
\textbf{Family 5:}~$(0,~0,~8,~0,~3)$ 
\begin{equation}\label{eq19}
    \begin{matrix*}[l]
         \alpha \beta ^7=ab^7, & 8\alpha \beta^3=8ab^3, & 3\alpha\beta -1=3ab.
    \end{matrix*}
\end{equation}
\begin{itemize}
    \item 
    (\ref{eq19}) : $6$ pairs.
\end{itemize}
\textbf{Family 15:}~$(0,~0,~0,~0,~0)$ 
\begin{equation}\label{eq20}
    \begin{matrix*}[l]
         \alpha \beta ^7=ab^7, & 12=2ab,
    \end{matrix*}
\end{equation}
\begin{equation}\label{eq21}
    \begin{matrix*}[l]
        \alpha \beta ^7=ab^7, &  12=6ab.
    \end{matrix*}
\end{equation}
\begin{itemize}
    \item 
    (\ref{eq20}) : $2$ pairs, (\ref{eq21}) : $2$ pairs
\end{itemize}
No additional pairs were found in the remaining families. Thus, we have found a total of $8+8+6+6+6+4=38$ pairs of $(\alpha,~\beta)$ from all the families resulting in total $38\times 13 \times 13=6422$ orthomorphism polynomials over $\mathbb{F}_{13},$ of which four are exceptional polynomials belonging to Family 15. These data match exactly with \cite[Table~3.2]{shallue}.

Subsequently, the same methodology is applied to $q=11,~17,~19,~23,~25~\text{and}~31,$ with minor modifications to the code \textbf{OP7(13)}. As in the case $q=13$, we list the number of pairs $(\alpha,~\beta)$ for each family (only when they exist) written in terms of the tuples as in \Cref{l4}  together with those associated systems of equations which are solvable over their respective fields for each $q\in\{11,~17,~19,~23,~25,~31\}.$ 
\vspace{0.3 cm}
\\
\underline{\Large \textbf{$q=11$~:}}
\vspace{0.2cm}
\begin{enumerate}[label=\textbf{(\arabic*)}]
    \item 
    $(1,~0,~5,~0,~2)$ 
\begin{equation}\label{eq23}
    \begin{matrix*}[l]
          \alpha \beta ^7=ab^7, & \alpha \beta ^5=ab^5, & 5\alpha \beta ^3=5ab ^3, & 2\alpha\beta-1=7ab,
    \end{matrix*}
\end{equation}
\begin{equation}\label{eq26}
    \begin{matrix*}[l]
          \alpha \beta ^7=ab^7, & \alpha \beta ^5=ab^5 & 5\alpha \beta ^3=5ab ^3, & 2\alpha\beta-1=9ab,
    \end{matrix*}
\end{equation}

\begin{itemize}
    \item 
    (\ref{eq23}) : $ 5$ pairs, (\ref{eq26}) : $ 5$ pairs.
    \end{itemize}
    \item
$(1,~0,~5,~0,~7)$ 
\begin{equation}\label{eq28}
    \begin{matrix*}[l]
         \alpha \beta ^7=ab^7, & \alpha \beta ^5=ab^5, & 5\alpha \beta ^3=5ab ^3, & 7\alpha\beta-1=2ab,
    \end{matrix*}
\end{equation}
\begin{equation}\label{eq32}
    \begin{matrix*}[l]
          \alpha \beta ^7=ab^7, & \alpha \beta ^5=ab^5 & 5\alpha \beta ^3=5ab ^3, & 7\alpha\beta-1=9ab,
    \end{matrix*}
\end{equation}

\begin{itemize}
    \item 
    (\ref{eq28}) : $5$ pairs,  (\ref{eq32}) : $5$ pairs.
    
    \end{itemize}
    \item 
    $(2,~0,~9,~0,~8)$ 
\begin{equation}\label{eq37}
    \begin{matrix*}[l]
          \alpha \beta ^7=ab^7, & 2\alpha \beta ^5=2ab^5, & 9\alpha \beta ^3=9ab^3, & 8\alpha\beta-1=9ab,
    \end{matrix*}
\end{equation}
\begin{equation}\label{eq39}
    \begin{matrix*}[l]
          \alpha \beta ^7=ab^7, & 2\alpha \beta ^5=2ab^5, & 9\alpha \beta ^3=9ab ^3, & 8\alpha\beta-1=6ab.
    \end{matrix*}
\end{equation}

\begin{itemize}
    \item 
    (\ref{eq37}) : $5$ pairs, (\ref{eq39}) : $5$ pairs.
    \end{itemize}

    \item 
    $(2,~0,~9,~0,~9)$ 

\begin{equation}\label{eq42}
    \begin{matrix*}[l]
          \alpha \beta ^7=ab^7, & 2\alpha \beta ^5=2ab^5, & 9\alpha \beta ^3=9ab ^3, & 9\alpha\beta-1=8ab,
    \end{matrix*}
\end{equation}
\begin{equation}\label{eq45}
    \begin{matrix*}[l]
          \alpha \beta ^7=ab^7, & 2\alpha \beta ^5=2ab^5, & 9\alpha \beta ^3=9ab ^3, & 9\alpha\beta-1=6ab.
    \end{matrix*}
\end{equation}

\begin{itemize}
    
    \item 
    (\ref{eq42}) : $5$ pairs, (\ref{eq45}) : $5$ pairs.
    \end{itemize}
    
    \item 
    $(1,~0,~5,~0,~9)$ 
\begin{equation}\label{eq46}
    \begin{matrix*}[l]
         \alpha \beta ^7=ab^7, & \alpha \beta ^5=ab^5, & 5\alpha \beta ^3=5ab ^3, & 9\alpha\beta-1=2ab,
    \end{matrix*}
\end{equation}
\begin{equation}\label{eq47}
    \begin{matrix*}[l]
          \alpha \beta ^7=ab^7, & \alpha \beta ^5=ab^5, & 5\alpha \beta ^3=5ab ^3, & 9\alpha\beta-1=7ab,
    \end{matrix*}
\end{equation}

\begin{itemize}
    \item 
    (\ref{eq46}) : $5$ pairs, (\ref{eq47}) : $5$ pairs.
   \end{itemize}

     \item 
     $(2,~0,~9,~0,~6)$ 

\begin{equation}\label{eq54}
    \begin{matrix*}[l]
          \alpha \beta ^7=ab^7, & 2\alpha \beta ^5=2ab^5, & 9\alpha \beta ^3=9ab ^3, & 6\alpha\beta-1=8ab,
    \end{matrix*}
\end{equation}
\begin{equation}\label{eq55}
    \begin{matrix*}[l]
          \alpha \beta ^7=ab^7, & 2\alpha \beta ^5=2ab^5, & 9\alpha \beta ^3=9ab^3, & 6\alpha\beta-1=9ab,
    \end{matrix*}
\end{equation}

\begin{itemize}
    \item 
    (\ref{eq54}) : $5$ pairs, (\ref{eq55}) : $5$ pairs.
    
    \end{itemize}
    
\end{enumerate}

    This leads us to the conclusion that there are total $60\times 11 \times 11 =7260$ orthomorphism polynomials over $\mathbb{F}_{11}.$ Our findings for $\mathbb{F}_{11}$ also match perfectly with \cite[Table~3.2]{shallue}. Among the $7260$ polynomials obtained, $10+10=20$ are exceptional coming from \textbf{(5)} and \textbf{(6)}.
    \vspace{0.5 cm}\\
\underline{\Large \textbf{$q=17$~:}}
\vspace{0.2cm}
\begin{enumerate}[label=\textbf{(\arabic*)}]
\item 
$(1,~0,~13,~0,~7)$ 

\begin{equation}\label{eq61}
    \begin{matrix*}[l]
          \alpha \beta ^7=ab^7, & \alpha \beta ^5=ab^5, & 13\alpha \beta ^3=13ab^3, & 7\alpha\beta-1=14ab,
    \end{matrix*}
\end{equation}
\begin{itemize}
    \item 
    (\ref{eq61}) : $8$ pairs.
    \end{itemize}

    \item 
    $(1,~0,~13,~0,~14)$ 

\begin{equation}\label{eq72}
    \begin{matrix*}[l]
          \alpha \beta ^7=ab^7, & \alpha \beta ^5=ab^5, & 13\alpha \beta ^3=13ab ^3, & 7\alpha\beta-1=7ab,
    \end{matrix*}
\end{equation}
\begin{itemize}
    
    \item 
    (\ref{eq72}) : $8$ pairs.
    \end{itemize}
    \end{enumerate}
    
We find exactly $16\times 17 \times 17=4624$ orthomoprphism polynomials over $\mathbb{F}_{17},$ consistent with Shallue and Wanless's tabulated data \cite[Table~3.2]{shallue}.
 \vspace{0.5 cm}\\
\underline{\Large \textbf{$q=19$~:}}
\vspace{0.2cm}
\begin{enumerate}[label=\textbf{(\arabic*)}]
    \item 
    $(0,~0,~0,~0,~16)$ 
\begin{equation}\label{eq82}
    \begin{matrix*}[l]
         \alpha \beta ^7=ab^7, & 16\alpha\beta-1=16ab,
    \end{matrix*}
\end{equation}
\begin{equation}\label{eq83}
    \begin{matrix*}[l]
         \alpha \beta ^7=ab^7, & 16\alpha\beta-1=0.
    \end{matrix*}
\end{equation}
\begin{itemize}
    \item 
    (\ref{eq82}) : $6$ pairs, (\ref{eq83}) : $3$ pairs.
    \end{itemize}
    \item 
    $(0,~0,~0,~0,~0)$ 
\begin{equation}\label{eq84}
    \begin{matrix*}[l]
         \alpha \beta ^7=ab^7, & 18=16ab.
    \end{matrix*}
\end{equation}
\begin{itemize}
    \item 
    (\ref{eq84}) : $3$ pairs.
    \end{itemize}
\end{enumerate}

We have found a total of $12\times 19\times 19=4332$ orthomorphism polynomials over $\mathbb{F}_{19}$ including $4$ exceptional polynomials. While Shallue et al. \cite{shallue} provided data only up to $\mathbb{F}_{17},$ our results for $\mathbb{F}_{19}$ serve as a natural extension of that tabulated data.
\vspace{0.5 cm}
\\
\underline{\Large \textbf{$q=25$~:}}
\vspace{0.2cm}\\
As arithmetic in $\mathbb{F}_{25}$ differs from that in prime fields such as $\mathbb{F}_{13},$ the previous SageMath code needs a little modification in the construction part of $\mathbb{F}_{25},$ which we leave to the reader. In what follows, $\theta=\theta^2+2$ in $\mathbb{F}_{25}$. 
 \begin{enumerate}[label=\textbf{(\arabic*)}]
     \item 
     $(0,~0,~0,~0,~\theta)$ 
\begin{equation}\label{eq85}
    \begin{matrix*}[l]
         \alpha \beta ^7=ab^7, & \theta\alpha\beta-1=\theta ab,
    \end{matrix*}
\end{equation}
\begin{equation}\label{eq86}
    \begin{matrix*}[l]
         \alpha \beta ^7=ab^7, & \theta\alpha\beta-1=\theta^5ab.
    \end{matrix*}
\end{equation}
\begin{equation}\label{eq87}
    \begin{matrix*}[l]
         \alpha \beta ^7=ab^7, & \theta\alpha\beta-1=0.
    \end{matrix*}
\end{equation}
\begin{itemize}
    \item 
    (\ref{eq85}) : $12$ pairs, (\ref{eq86}) : $16$ pairs, (\ref{eq87}) : $4$ pairs.
    \end{itemize}
 \item 
 $(0,~0,~0,~0,~\theta^5)$ 
\begin{equation}\label{eq88}
    \begin{matrix*}[l]
         \alpha \beta ^7=ab^7, & \theta^5\alpha\beta-1=\theta ab,
    \end{matrix*}
\end{equation}
\begin{equation}\label{eq89}
    \begin{matrix*}[l]
         \alpha \beta ^7=ab^7, & \theta^5\alpha\beta-1=\theta^5ab.
    \end{matrix*}
\end{equation}
\begin{equation}\label{eq90}
    \begin{matrix*}[l]
         \alpha \beta ^7=ab^7, & \theta^5\alpha\beta-1=0.
    \end{matrix*}
\end{equation}
\begin{itemize}
    \item 
    (\ref{eq88}) : $16$ pairs, (\ref{eq89}) : $12$ pairs, (\ref{eq90}) : $4$ pairs.
    \end{itemize}
\item 
$(\theta,~0,~\theta^2,~0,~0)$ 
\begin{equation}\label{eq91}
    \begin{matrix*}[l]
         \alpha \beta ^7=ab^7, & \theta\alpha\beta^5=\theta ab^5, & \theta^2\alpha\beta^3=\theta^2ab^3, & 4=4\theta^3ab,
    \end{matrix*}
\end{equation}
\begin{itemize}
    \item 
    (\ref{eq91}) : $12$ pairs.
    \end{itemize}
    \item 
    $(0,~0,~0,~0,~0)$ 
\begin{equation}\label{eq93}
    \begin{matrix*}[l]
         \alpha \beta ^7=ab^7,  & 4=\theta ab,
    \end{matrix*}
\end{equation}
\begin{equation}\label{eq94}
    \begin{matrix*}[l]
         \alpha \beta ^7=ab^7,  & 4=\theta^5ab,
    \end{matrix*}
\end{equation}
\begin{itemize}
    \item 
    (\ref{eq93}) : $4$ pairs, (\ref{eq94}) : $4$ pairs.
    \end{itemize}
\item 
$(\theta,~0,~\theta^2,~0,~4\theta^3)$ 
\begin{equation}\label{eq95}
    \begin{matrix*}[l]
         \alpha \beta ^7=ab^7, & \theta \alpha\beta^5=\theta ab^5, & \theta^2\alpha\beta^3=\theta^2ab^3, & 4\theta^3\alpha\beta-1=0,
    \end{matrix*}
\end{equation}
\begin{itemize}
    \item 
    (\ref{eq95}) : $12$ pairs.
    \end{itemize}
 \end{enumerate}

Therefore, the total number of orthomorphism polynomials of degree $7$ over $\mathbb{F}_{25}$ is $96\times 25\times 25=60,000,$ among which $20$ are exceptional polynomials.
 \vspace{0.5 cm}\\
\underline{\Large \textbf{$q=23$}~and \textbf{$q=31$:}}
\vspace{0.2cm}\\
For both $\mathbb{F}_{23}$ and $\mathbb{F}_{31},$ no pairs of $(\alpha,~\beta)$ exist for any of their respective families, implying the non-existence of orthomorphism polynomials of degree $7$ over these fields.
\vspace{0.2 cm}\\
We now summarize all of the preceding findings in the form of the following theorem.
\begin{theorem}\label{thm3.1}
    Let $\mathbb{F}_q$ be a finite field of characteristic $p\notin\{2,~3,~7\}$ order $q> 7.$ A non-exceptional orthomorphism polynomial $f$ of degree $7$ exists over $\mathbb{F}_q$ if and only if $q\in\{11,~13,~17,~19,~25\}.$ These polynomials, along with all the exceptional polynomials of degree $7$ over these fields are listed below. The complete collection can be obtained by taking $f(x+\gamma)+\delta,$ where $\gamma,~\delta\in\mathbb{F}_q.$
    \begin{itemize}
        \item 
        $q=11$
        \begin{enumerate}[label=\textup{(\arabic*)}]
        \item 
        $\alpha\beta^7x^7+\alpha\beta^5x^5+5\alpha\beta^3x^3+2\alpha\beta x,$ \\
        where $(\alpha,~\beta)\in\{(1,~2),~(2,~1),~(3,~8),~(4,~6),~(5,~7),~(1,~3),~(2,~7),~(3,~1),~(4,~9),\\~(5,~5)\}.$
        \item 
        $\alpha\beta^7x^7+\alpha\beta^5x^5+5\alpha\beta^3x^3+7\alpha\beta x,$\\ where $(\alpha,~\beta)\in\{(1,~9),~(2,~10),~(3,~3),~(4,~5),~(5,~4),~(1,~5),~(2,~8),~(3,~9),~(4,~4),\\~(5,~1)\}.$
        \item 
        $\alpha\beta^7x^7+2\alpha\beta^5x^5+9\alpha\beta^3x^3+8\alpha\beta x,$\\
        where $(\alpha,~\beta)\in\{(1,~10),~(2,~5),~(3,~7),~(4,~8),~(5,~2),~(1,~6),~(2,~3),~(3,~2),~(4,~7),\\~(5,~10)\}.$
        \item 
        $\alpha\beta^7x^7+2\alpha\beta^5x^5+9\alpha\beta^3x^3+9\alpha\beta x,$\\
        where $(\alpha,~\beta)\in\{(1,~1),~(2,~6),~(3,~4),~(4,~3),~(5,~9),~(1,~4),~(2,~2),~(3,~5),~(4,~1),\\~(5,~3)\}.$
        \item 
        $\alpha\beta^7x^7+\alpha\beta^5x^5+5\alpha\beta^3x^3+9\alpha\beta x,$\\
        where $(\alpha,~\beta)\in\{(1,~8),~(2,~4),~(3,~10),~(4,~2),~(5,~6),~(1,~6),~(2,~3),~(3,~2),~(4,~7),\\~(5,~10)\}.$
        \item 
        $\alpha\beta^7x^7+2\alpha\beta^5x^5+9\alpha\beta^3x^3+6\alpha\beta x,$\\
        where $(\alpha,~\beta)\in\{(1,~5),~(2,~8),~(3,~9),~(4,~4),~(5,~1),~(1,~7),~(2,~9),~(3,~6),~(4,~10),\\~(5,~8)\}.$
    \end{enumerate}
    \item 
    $q=13$
    \begin{enumerate}[label=\textup{(\arabic*)}]
     \item 
        $\alpha\beta^7x^7+2\alpha\beta x,$\\
        where $(\alpha,~\beta)\in\{(2,~5),~(1,~10),~(1,~3),~(1,~5),~(2,~9),~(2,~8),~(2,~10),~(1,~7)\}.$
         \item 
        $\alpha\beta^7x^7+6\alpha\beta x,$\\
        where $(\alpha,~\beta)\in\{(5,~1),~(1,~10),~(1,~5),~(2,~5),~(2,~6),~(1,~12),~(2,~12),~(1,~11)\}.$
         \item 
        $\alpha\beta^7x^7+2\alpha\beta^3 x^3+8\alpha\beta x,$\\
        where $(\alpha,~\beta)\in\{(5,~7),~(3,~3),~(2,~11),~(6,~8),~(1,~9),~(4,~12)\}.$
        \item 
        $\alpha\beta^7x^7+4\alpha\beta^3 x^3+4\alpha\beta x,$\\
        where $(\alpha,~\beta)\in\{(5,~1),~(3,~6),~(2,~9),~(6,~3),~(1,~5),~(4,~11)\}.$
        \item 
        $\alpha\beta^7x^7+8\alpha\beta^3 x^3+3\alpha\beta x,$\\
        where $(\alpha,~\beta)\in\{(5,~10),~(3,~8),~(2,~12),~(6,~4),~(1,~11),~(4,~6)\}.$
        \item
        $\alpha\beta^7x^7,$\\
        where $(\alpha,~\beta)\in\{(1,~6),~(1,~7),~(1,~2),~(1,~11)\}.$
    \end{enumerate}
    \item
    $q=17$
    \begin{enumerate}[label=\textup{(\arabic*)}]
         \item 
        $\alpha\beta^7x^7+\alpha\beta^5 x^5+13\alpha\beta^3x^3+7\alpha\beta x,$\\
        where $(\alpha,~\beta)\in\{(1,~12),~(2,~6),~(3,~4),~(4,~3),~(5,~16),~(6,~2),~(7,~9),~(8,~10)\}.$
        \item 
        $\alpha\beta^7x^7+\alpha\beta^5 x^5+13\alpha\beta^3x^3+14\alpha\beta x,$\\
        where $(\alpha,~\beta)\in\{(1,~5),~(2,~11),~(3,~13),~(4,~14),~(5,~1),~(6,~15),~(7,~8),~(8,~7)\}.$
        \end{enumerate}
         \item
    $q=19$
    \begin{enumerate}[label=\textup{(\arabic*)}]
         \item 
        $\alpha\beta^7x^7+16\alpha\beta x,$\\
        where $(\alpha,~\beta)\in\{(2,~13),~(4,~14),~(1,~18),~(4,~16),~(2,~9),~(1,~7),~(4,~11),~(1,~6),\\~(2,~3)\}.$
        \item 
        $\alpha\beta^7x^7,$\\
        where $(\alpha,~\beta)\in\{(1,~13),~(1,~10),~(1,~15)\}.$
        \end{enumerate}
        \item 
        $q=25$
        \begin{enumerate}[label=\textup{(\arabic*)}]
        \item 
        $\alpha\beta^7x^7+\theta\alpha\beta x,$\\
        where $(\alpha,~\beta)\in\{(4\theta+3,~4\theta+2),~(\theta+3,,~3\theta+1),~(2\theta+2,~\theta),~(2\theta+2,~4\theta),\\~(4\theta+3,~2\theta+1),~(\theta,~2\theta+2),~(\theta,~3\theta+3),~(2\theta+2,~2\theta),~(\theta+3,~4\theta+3),\\~(\theta+3,~\theta+2),~(\theta,~4\theta+4),~(4\theta+3,~\theta+3),~(\theta,~\theta+3),~(2\theta+2,~\theta),\\~(4\theta+3,~2\theta+4),~(\theta+3,~2),~(\theta+3,~\theta),~(\theta,~3\theta+1),~(2\theta+2,~3\theta+4),\\~(4\theta+3,~4\theta+1),~(\theta+3,~2\theta+1),~(\theta,~2\theta+3),~(2\theta+2,~2\theta+2),~(4\theta+3,~1),\\~(\theta+3,~\theta+1),~(\theta,~2\theta),~(2\theta+2,~2\theta+3),~(4\theta+3,~4\theta),~(\theta+3,~2\theta+4),\\~(4\theta+3,~3\theta+4),~(2\theta+2,~3\theta),~(\theta,~\theta+1)\}.$
            \item 
            $\alpha\beta^7x^7+\theta^5\alpha\beta x,$\\
            where $(\alpha,~\beta)\in\{(4\theta+3,~2\theta+4),~(\theta,,~4\theta+2),~(2\theta+2,~\theta+1),\\~(\theta+3,~4\theta+2),~(2\theta+2,~3\theta+4),~(\theta+3,~4\theta),~(\theta,~\theta),~(4\theta+3,~4\theta),~(\theta,~2\theta+3),\\~(4\theta+3,~4),~(\theta+3,~1),~(2\theta+2,~4),~(\theta+3,~\theta+1),~(2\theta+2,~3\theta+2),\\(4\theta+3,~2\theta+3),~(\theta,~\theta+2),~(\theta+3,~4\theta+1),~(2\theta+2,~2\theta+4),\\(4\theta+3,~4\theta+4),~(4\theta+3,~2\theta+2),~(\theta,~1),~(2\theta+2,~3\theta+1),~(2\theta+2,~4\theta+3),\\(\theta+3,~2\theta+3),~(\theta,~4),~(\theta,~2),~(4\theta+3,~\theta+1),~(\theta+3,~3\theta+2),\\(4\theta+3,~3\theta+3),~(2\theta+2,~\theta+2),~(\theta,~3),~(\theta+3,~\theta+4)\}.$
            \item 
            $\alpha\beta^7x^7+\theta\alpha\beta^5x^5+\theta^2\alpha\beta^3 x^3,$\\
            where $(\alpha,~\beta)\in\{(\theta,~3\theta+4),~(\theta+3,~3\theta),~(4\theta+3,~3),~(2\theta+2,~\theta+4),\\(4\theta+1,~3\theta+3),~(2,~\theta+2),~(2\theta,~4\theta+2),~(2\theta+1,~4\theta),~(3\theta+1,~4),\\(4\theta+4,~3\theta+2),~(3\theta+2,~4\theta+4),~(4,~3\theta+1)\}.$
            \item 
            $\alpha\beta^7x^7,$\\
            where $(\alpha,~\beta)\in\{(\theta,~4\theta+4),~(\theta,~3\theta+3),~(\theta,~\theta+1),~(\theta,~2\theta+2),~(\theta,~2),~(\theta,~4),~(\theta,~3),\\(\theta,~1)\}.$
            \item $\alpha\beta^7x^7+\theta\alpha\beta^5x^5+\theta^2\alpha\beta^3x^3+4\theta^3\alpha\beta x,$\\
            where $(\alpha,~\beta)\in\{(\theta,~2\theta+1),~(\theta+3,~2\theta),~(4\theta+3,~2),~(2\theta+2,~4\theta+1),\\(4\theta+1,~2\theta+2),~(2,~4\theta+3),~(2\theta,~\theta+3),~(2\theta+1,~\theta),~(3\theta+1,~1),\\(4\theta+4,~2\theta+3),~(3\theta+2,~\theta+1),~(4,~2\theta+4)\},$
            \end{enumerate}
            with $\theta=\theta^2+2$ in $\mathbb{F}_{25}.$
    \end{itemize}
    
\end{theorem}

\begin{corollary}
    Any orthomorphism polynomial of degree $7$ over $\mathbb{F}_{17}$ is non-exceptional.
\end{corollary}
\begin{proof}
    It follows from the fact that a linear transformation of a non-exceptional polynomial remains non-exceptional.
\end{proof}
\begin{theorem}
    There does not exist any orthomorphism polynomial of degree $7$ over $\mathbb{F}_{23}$ and $\mathbb{F}_{31}.$
\end{theorem}
\vspace{0.5 cm}
\subsection{Non-exceptional OPs over fields of characteristic $p\in\{2,\,3,\,7\}$}
We now investigate the existence of non-exceptional OPs over fields of characteristic $2,\,3$ and $7.$ It is known that the reduced degree of an OP over $\mathbb{F}_q$ is at most $q-3$ \cite{CPP1, wan}, where the reduced degree of a polynomial $f$ is the degree of its remainder modulo $x^q-x.$ Consequently, an OP of reduced degree $7$ can exist over $\mathbb{F}_q$ only when $q\geq 11.$ Note that, when $q>7,$ any degree $7$ polynomial  over $\mathbb{F}_q$ has reduced degree $7.$ Therefore to classify OPs of degree $7$ over fields of characteristic $2,\,3$ and $7$ it suffices to consider fields of order greater than or equal to  $16,\,27$ and $49$ respectively.

\vspace{0.3 cm}
\subsubsection{OPs over fields of even characteristic:}
Shallue et al. \cite{shallue} established by computer enumeration that that no orthomorphism polynomials of degree $7$ exist over $F_{16}.$ The same conclusion can also be obtained by applying our technique applied for earlier values of $q$ together with the known results of Li et al. \cite[Theorem~4.4]{li} on degree $7$ PPs over fields of characteristic $2.$ We state below a modified version of their theorem, as it appeared in \cite[Lemma 5]{xfan}, which is more suitable for our purposes.

\begin{lemma}\cite[Theorem~4.4]{li}\label{l5}
    Let $q=2^m> 8.$ Then we have the following:
    \begin{enumerate}
        \item 
        There exists a PP of degree $7$ over $\mathbb{F}_q$ if and only if $3\nmid m$ and in such case, each exceptional polynomial over $\mathbb{F}_q$ is linearly related to $x^7$ or $x^7+x^5+x.$
        \item 
        A non-exceptional PP of degree $7$ exists over $\mathbb{F}_q$  if and only if  $m=4,$ and it is linearly related to one of the following \\
        $x^7 +x^4 +x,\;x^7+x^5+x^4,\;x^7+x^5+ax^4+a^{14}x^3+a^{12}x^2+a^8x,\;
        x^7+x^5 +a^7x^4 +a^5x^2 +a^3x,\;x^7 +x^5+a^5x^4 +a^2x^3 +a^{12}x^2 +a^5x,$ where $a^4+a+1=0$ in $\mathbb{F}_{2^4}.$
        
    \end{enumerate}

    \end{lemma}
    It follows immediately from (1) of Lemma~\ref{l5} that there are no exceptional OPs of degree $7,$ a fact also observed by Shallue et al \cite{shallue}. It therefore remains to examine the non-exceptional PPs of degree $7$ over $\mathbb{F}_{16}$ and determine which of them are OPs. Using our method, the problem then reduces to finding $(\alpha,\,\beta)\in(\mathbb{F}_{16}^*)^2$ such that $\alpha f(\beta x)-x$ is linearly related to $x^7,\, x^7+x^5+x$ or $g(x),$ where $f$ and $g$ are listed in (2) of Lemma~\ref{l5}. To determine such pairs, we solve a system of equations over $\mathbb{F}_{16}^*$ using SageMath following the same approach as in the preceding cases and each system is found to be inconsistent over $\mathbb{F}_{16}^{*}.$ Thus, the following theorem follows.
    \begin{theorem}\label{thm 3.4}
        Let $\mathbb{F}_q$ be a finite field of even characteristic. Then there does not exist any orthomorphism polynomial of degree $7$ over $\mathbb{F}_q.$
    \end{theorem}

    \subsubsection{Non-exceptional OPs over fields of characteristic $3$ and $7$}

 The following Lemma from \cite{xfan} provides a classification of PPs over these fields.
\begin{lemma}\cite[Theorem~14,  Theorem~16]{xfan} \label{thmm}(1) Each PP over degree $7$ over $\mathbb{F}_{27}$ is linearly related to either exceptional $x^7$ or non-exceptional $x^7-x^3+x.$ If $k\geq 4,$ all PPs of degree $7$ over $\mathbb{F}_{3^k}$ are exceptional.\\
(2) Let $\theta$ be a root of $x^2+6x+3$ in $\mathbb{F}_{49}.$ All non-exceptional PPs of degree $7$  over $\mathbb{F}_{49}$ are linearly related to $$x^7+\theta x^5+2\theta^2x^3+4\theta^3x$$ and each exceptional polynomial of degree $7$ over $\mathbb{F}_{49}$ is linearly related to exactly one of the following polynomials:
\begin{itemize}
    \item 
    $x^7+ax$ with $a\in\{0,~\theta,~\theta^2,~\theta^3,~\theta^4,~\theta^5\},$
    \item 
    $x^7+x^4+2x,$
    \item 
    $x^7+\theta^2x^4+2\theta^4x,$ and
    \item 
    $x^7+\theta x^5+5\theta^2x^3+6\theta^3x.$
\end{itemize}
(3) Each PP of degree $7$ over $\mathbb{F}_{7^k},~k\geq 3$ is exceptional.
\end{lemma}
As a direct consequence of Lemma~\ref{thmm}, we have the following result.
\begin{theorem}\label{thm 3.5}
    There does not exist any non-exceptional OP of degree $7$ over $\mathbb{F}_q,$ when $q=3^m,\,m>3$ or $q=7^m,\,m>2.$ 
\end{theorem}
To determine the existence of non-exceptional OPs of degree $7$ over $\mathbb{F}_{27}$ and $\mathbb{F}_{49},$ we again use \Cref{thmm} and follow a similar approach as used earlier and obtain the following results. 
\begin{theorem}\label{cor3.5}
    (1) There exists no orthomorphism polynomial of degree $7$ over $\mathbb{F}_{27}$, either non-exceptional or exceptional. \\
    (2) Any orthomorphism polynomial of degree $7$ over $\mathbb{F}_{49}$ is exceptional and there are a total of $3,937,640$ orthomorphism polynomials of degree $7$ over $\mathbb{F}_{49}.$
    \end{theorem}
    \begin{remark}\label{rmk 3}
    All such orthomorphism polynomials over $\mathbb{F}_{49}$ are of the form $f(x+\gamma)+\delta$ with $\gamma,~\delta \in \mathbb{F}_{49}$;~$f(x)=\alpha\beta^7x^7+a \alpha\beta x,~a\in\{0,~\theta,~\theta^2,~\theta^3,~\theta^4,~\theta^5\}$ where $\theta$ is a root of $x^2+6x+3$ in $\mathbb{F}_{49}$ and $(\alpha,~\beta)$ takes specific values in $\mathbb{F}_{49}^2$ that can be obtained by the method we previously used for $\mathbb{F}_{25}.$ The total number of pairs of $(\alpha,~\beta)$ is $40$ when $a=0$ and $320$ for each of the values of $a=\theta,~\theta^2,~\theta^3,~\theta^4$ and $\theta^5.$ Owing to the large number of pairs $(\alpha,~\beta)$  ($1640$ in total), we list only those corresponding to $a=0$ for illustration and they are;\\
     $(\theta,~4\theta+1),~(\theta,~5\theta,+1),~(\theta,~4\theta+2),~(\theta,~5),~(\theta,~3\theta+6),~(\theta,~2\theta+6),~(\theta,~3\theta+5),~(\theta,~2),\\(\theta,~2\theta+2),~(\theta,~2\theta+3),~(\theta,~4\theta),~(\theta,~3\theta+4),~(\theta,~5\theta+5),~(\theta,~5\theta+4),~(\theta,~3\theta),\\(\theta,~4\theta+3),~(\theta, ~4\theta+ 5),~(\theta, ~ 6\theta + 3),~ (\theta,~ 4),~ (\theta,~ \theta+ 2),~(\theta, ~3\theta+ 2),~(\theta, ~\theta+ 4),~(\theta,~ 3),\\(\theta, ~ 6\theta + 5),~(\theta,~ 3\theta+ 1),~(\theta,~ 6\theta),~(\theta, ~\theta+ 6),~(\theta,~ 4\theta + 4),~(\theta, ~4\theta+ 6),~(\theta,~ \theta),~(\theta, ~ 6\theta + 1),\\(\theta,~ 3\theta+ 3),~(\theta,~ 2\theta+ 1),~(\theta,~ 6),~(\theta, ~5\theta+ 3),~(\theta,~ \theta+3),~(\theta, ~5\theta+ 6),~(\theta,~ 1),~(\theta,~ 2\theta+ 4),~(\theta,~6\theta + 4),$ where $\theta^2+6\theta+3=0$ in $\mathbb{F}_{49}.$
     \end{remark}
\vspace{0.5 cm}
Theorem~\ref{thm 3.4}, Theorem~\ref{thm 3.5} and \Cref{cor3.5} together yield the following broader version of \Cref{thm3.1}.
\begin{theorem}
    For any finite field $\mathbb{F}_q$ of order greater than $7,$ a non-exceptional orthomorphism polynomial of degree $7$ exists if and only if $q\in\{11,~13,~17,~19,~25\}.$
\end{theorem}
\noindent
\subsection{\boldmath The case $q\equiv \pm 1 \pmod 7,\,q\notin\{13,\,27\}$}
It is known that when $q\equiv 1 \pmod 7,$ there is no degree $7$ PPs over $\mathbb{F}_q,$ and consequently no degree $7$ OPs. For the case $q\equiv 6\pmod 7,$ the following result appears in \cite{xfan}.
\begin{lemma} \label{l6}\cite{xfan}
    Let $\mathbb{F}_q$ be a finite field, where $q\equiv 6 ~(\textup{mod}~7)$ and $q\notin\{13,~27\}.$ Then any polynomial $f$ of degree $7$ is a PP over $\mathbb{F}_q$ if and only if it is linearly related to $x^7.$
\end{lemma}
As a consequence of Theorem~\ref{cor3.5} and Lemma~\ref{l6} we have the following theorem.
\begin{theorem}
    No orthomorphism polynomials of degree $7$ exist over $\mathbb{F}_q;$ where $q\equiv 6 ~(\textup{mod}~7),~q\neq 13.$
\end{theorem}
\subsection{Exceptional OPs}As the classification of all non-exceptional OPs of degree $7$ is complete, we now turn our attention to the remaining values of $q,$ satisfying $q\equiv r \pmod 7,\,r\in\{0,\,2,\,3,\,4,\,5\},$ where only exceptional OPs can occur. 

\subsubsection{Fields of characteristic not equal to $7$}The cases with $r\neq 0$ and $q\leq 31$ have already been settled earlier. For $r\neq 0$ and $q>31$, Fan pointed out that any PP over such a finite field must be linearly related to one of the following exceptional polynomials \cite[Page no. 19]{xfan}: $$f_1(x)=x^7,\,f_2(x)=x^7+x^5+2.7^{-1}x^3+7^{-2}x,\,f_3(x)=x^7+ex^5+2.7^{-1}e^2x^3+7^{-2}e^3x,$$ where $e$ is a generator of $\mathbb{F}_q^*.$ Using similar method, we check for the values of $(\alpha,\,\beta)\in(\mathbb{F}_q^*)^2$ such that $\alpha f_i(\beta x)-x,\,i\in\{1,\,2,\,3\}$ is a PP over $\mathbb{F}_q.$ 

If $\alpha f_1(\beta x)-x$ is a PP over $\mathbb{F}_q$ for some $(\alpha,\,\beta)\in(\mathbb{F}_q^*)^2$ then we have $$\alpha f_1(\beta x)-x=af_i(bx+c)+d,$$ for some $i\in\{1,\,2,\,3,\}$  and $(a,\,b,\,c,\,d)\in(\mathbb{F}_q^*,\,\mathbb{F}_q^*,\,\mathbb{F}_q,\,\mathbb{F}_q).$ Comparing the coefficients of all powers of $x$ on both sides, we get that $c=d=0$ and $i=1,$ which then becomes a contradiction.

Let $\alpha f_2(\beta x)-x$ be a PP over $\mathbb{F}_q.$ Then we must have $$\alpha f_2(\beta x)-x=af_i(bx+c)+d,$$ for some $i\in\{1,\,2,\,3,\}$  and $(a,\,b,\,c,\,d)\in(\mathbb{F}_q^*,\,\mathbb{F}_q^*,\,\mathbb{F}_q,\,\mathbb{F}_q).$ If $i=1,$ then we have $\alpha\beta^7=ab^7,\,0=c=d,$ and $-1=0,$ which is a contradiction. If $i=2,$ then again we arrive at the contradiction $-1=0.$ Finally, for $i=3,$ we have the following set of equations: $$\alpha\beta^7=ab^7,\,\alpha\beta^5=eab^5,\,\alpha\beta^3=e^2ab^3,\,7^{-2}\alpha\beta-1=7^{-2}e^3ab.$$ 
From the first two equations, we have $\beta^2=e^{-1}b^2.$ Substituting this in the third equation, we get $\alpha \beta=e^3ab$ and it contradicts the last equation. 

Similarly, one can show that the same conclusion holds for $\alpha f_3(\beta x)-x$ also. It then follows that $\alpha f_i(\beta x)-x,\,i\in\{1,\,2,\,3\}$ is never PP over $\mathbb{F}_q$ for any $(\alpha,\,\beta)\in(\mathbb{F}_q^*)^2,$ and consequently we can conclude that there does not exist any OPs over $q\equiv r \pmod 7,\,r\in\{2,\,3,\,4,\,5\}$ and $q>31.$  Therefore, combining this fact with the results established in the preceding subsections, we obtain the following theorem.
\begin{theorem}
    Let $q=p^m>7$ and $p\neq 7.$ Then an orthomorphism polynomial of degree $7$ exists over $\mathbb{F}_q$ if and only if $q\in\{11,\,13,\,17,\,19,\,25\}.$
\end{theorem}
\subsubsection{Fields of characteristic $7$} Let $q=7^m>7.$ From Theorem~\ref{thm 3.5} and Theorem~\ref{cor3.5}, it is clear that any OP over $\mathbb{F}_q$ is exceptional. Moreover, we have given an explicit characterization of all OPs over $\mathbb{F}_{49}$ (see Remark~\ref{rmk 3}).
\begin{question}
    Characterize all orthomorphism polynomials over $\mathbb{F}_{7^m},\,m\geq 3.$
\end{question}

\section{Concluding Remarks}
We have determined all orthomorphism polynomials of degree $7$ over finite fields of all characteristics except $7$ and over $\mathbb{F}_{49},$ thereby extending the existing classification of such polynomials \cite{CPP1, shallue}. Furthermore, we have completely classified all non-exceptional orthomorphism polynomials over finite fields. To complete the classification of all orthomorphism polynomials of degree $7$, it suffices to examine the finite fields $\mathbb{F}_{7^r},~r\geq 3$ only. Note that, over this fields, only exceptional OPs will be obtained.
\section{Acknowledgements}
The authors would like to express their sincere gratitude to Prof. Ian  M 
Wanless, Monash University, Melbourne, Australia for his encouragement to pursue research in this area and for providing valuable related references. Bhitali Kousik also acknowledges financial support from the DST-INSPIRE Fellowship, Government of  India (INSPIRE Reg. No. IF230368).
\section{Declarations}
\noindent
\textbf{Conflict of Interest} The authors declare no competing interests.\\
\textbf{Ethical Approval} Not applicable.\\
\textbf{Data Availability} Not applicable.


\begin{thebibliography}{999}
\bibitem{evans}
Evans, A. B. \emph{Orthogonal Latin Squares Based on Groups}, Developments in Mathematics, vol. 57, Springer (Cham, Switzerland), 2018.
\bibitem{evans2}
Evans, A. B. \emph{Orthomorphism Graphs of Groups}, Springer, 1992.
\bibitem{xfan}
 Fan, X. A classification of permutation polynomials of degree $7$ over finite fields, \emph{Finite Fields Appl.}, \textbf{59} (2019) 1-21.
 \bibitem{li}
 Li, J., Chandler, D. B., and Xiang, Q. Permutation polynomials of degree 6 or 7 over finite fields of characteristic 2, \emph{Finite Fields Appl.}, \textbf{16} (2010) 406-419.
 \bibitem{mann}
 Mann, H. B. The construction of orthogonal Latin Squares, \emph{Ann. Math. Stat.}, \textbf{13} (1942) 418-423.
 \bibitem{fried}
 Fried, M. D., Guralnick, R. and Saxl, J. Schur covers and Carlitz’s conjecture, \emph{Isr. J. Math.},  \textbf{82} (1993) 157–225.
 \bibitem{muller}
M\"{u}ller, P. A Weil-bound free proof of Schur's conjecture, \emph{Finite Fields Appl.}, \textbf{3} (1997) 25-32.
\bibitem{muratovic}
Muratovi{\'c}-Ribi{\'c}, A. and Pasalic, E. A note on complete polynomials over finite fields and their applications in cryptography, \emph{Finite Fields Appl.}, \textbf{25} (2024) 306-315.
\bibitem{CPP1}
Neiderreiter, H., and Robinson, K. H. Complete Mappings of Finite Fields, \emph{J. Aust. Math. Soc.}, \textbf{33} (1982) 197-212.
\bibitem{schulz}
Schulz, R. -H. On check digit systems using anti-symmetric mappings, in: Alth{\"o}fer, I., et al. \emph{Numbers, Information and Complexity}, Springer, Boston, (2000) 295-310.
\bibitem{winterhof}
Shaheen, R., and Winterhof, A. Permutations of finite fields for check digit systems, \emph{Des. Codes Cryptography.}, \textbf{57} (2010) 361-371.
\bibitem{shallue}
Shallue, C. J., and  Wanless, I. M. Permutation polynomials and orthomorphism polynomials of degree six, \emph{Finite Fields Appl.}, \textbf{20} (2013) 84-92.
\bibitem{stanica}
St{\u{a}}nic{\u{a}}, P., Gangopadhyay, S., Chaturvedi, A., Gangopadhyay, A. K., and Maitra, S. Investigations on bent and negabent functions via the nega-Hadamard transform, \emph{IEEE Trans. Inf. Theory}, \textbf{58} (2012) 4064-4072.
\bibitem{wanless2}
Stones, D. S., and  Wanless, I. M. Compound orthomorphism of the cyclic group, \emph{Finite Fields Appl.}, \textbf{16} (2010) 277-289.
\bibitem{wanless3}
Wanless, I. M. Atomic Latin Squares based on Cyclotomic Orthomorphisms, \emph{Electron. J. Comb.}, \textbf{12} (2005) R22.
\bibitem{wan}
Wan, D. On a problem of Niederreiter and Robinson about finite fields, \emph{J. Aust. Math. Soc.}, \textbf{41} (1986) 336-338.
\bibitem{wanless}
Wanless, I. A. Diagonally cyclic latin squares, \emph{Eur. J. Comb.}, \textbf{25} (2004) 393-413.
\bibitem{winterhof2}
Winterhof, A. Generalizations of complete mappings of finite fields and some applications, \emph{J. Symb. Comput.}, \textbf{64} (2014) 42-52.
\end{thebibliography}
\end{document}